\documentclass[12pt]{amsart}
\usepackage[margin=1in]{geometry}
\usepackage{amssymb,amsmath,amsthm,hyperref,mathrsfs,multirow,xcolor}
\usepackage{tikz}
\usetikzlibrary{arrows,matrix,positioning}

\newtheorem{theorem}{Theorem}

\newtheorem{conj}[theorem]{Conjecture}

\newtheorem*{theoremNoNum}{Theorem}

\newtheorem*{propNoNum}{Proposition}

\theoremstyle{definition}
\newtheorem{definition}{Definition}
\newtheorem{lemma}{Lemma}
\newtheorem{example}{Example}

\theoremstyle{remark}
\newtheorem*{remark}{Remark}

\allowdisplaybreaks

\begin{document}
\newcommand\mylabel[1]{\label{#1}}
\newcommand{\beqs}{\begin{equation*}}
\newcommand{\eeqs}{\end{equation*}}
\newcommand{\beq}{\begin{equation}}
\newcommand{\eeq}{\end{equation}}
\newcommand\eqn[1]{(\ref{eq:#1})}
\newcommand\exam[1]{\ref{exam:#1}}
\newcommand\thm[1]{\ref{thm:#1}}
\newcommand\lem[1]{\ref{lem:#1}}
\newcommand\propo[1]{\ref{propo:#1}}
\newcommand\corol[1]{\ref{cor:#1}}
\newcommand\sect[1]{\ref{sec:#1}}
\newcommand\subsect[1]{\ref{subsec:#1}}

\newcommand{\Z}{\mathbb Z}
\newcommand{\N}{\mathbb N}
\newcommand{\R}{\mathbb R}
\renewcommand{\P}{\mathbb P}
\newcommand{\G}{\mathbb G}
\newcommand{\C}{\mathbb C}
\newcommand{\Q}{\mathbb Q}
\newcommand{\cA}{\mathcal A}
\newcommand{\cMT}{\mathcal{MT}}
\newcommand{\cC}{\mathcal C}
\newcommand{\cD}{\mathcal D}
\newcommand{\cS}{\mathcal S}
\newcommand{\cO}{\mathcal O}
\newcommand{\cU}{\mathcal U}
\newcommand{\cG}{\mathcal G}
\newcommand{\cH}{\mathcal H}
\newcommand{\cL}{\mathcal L}
\newcommand{\cZ}{\mathcal Z}
\newcommand{\SL}{\mathrm{SL}}
\newcommand{\gr}{\mathrm{gr}}
\newcommand{\dR}{\mathrm{dR}}
\newcommand{\depth}{\mathfrak{D}}
\newcommand{\ev}{\mathrm{ev}}
\newcommand{\even}{\mathrm{even}}
\newcommand{\odd}{\mathrm{odd}}
\renewcommand{\o}{\mathbf{o}}
\newcommand{\e}{\mathbf{e}}
\renewcommand{\a}{\mathfrak{a}}
\newcommand{\m}{\mathfrak{m}}
\renewcommand{\l}{\mathfrak{l}}

\newcommand{\td}{{\color{red}To do}}

\title[Double Zeta Values relative to $\mu_N$, Hecke Operators, and Newforms]
{Connections between Double Zeta Values relative to $\mu_N$, Hecke Operators $T_N$, and Newforms of Level $\Gamma_0(N)$ for $N=2,3$}

\author{Ding Ma}
\address{Department of Mathematics, University of Arizona, Tucson, AZ. 85721}
\email{martin@math.arizona.edu}

\date{\today}

\begin{abstract}
In this paper, we will study various connections between double zeta values relative to $\mu_N$, Hecke operators $T_N$, and newforms of level $\Gamma_0(N)$ for $N=2,3$. Those various connections generalize the well-known of Baumard and Schneps in \cite{BS}. We also give a conjecture about the Eichler-Shimura-Manin correspondence between $\cS_k^{\textrm{new}}(\Gamma_0(2))^{\pm}$, $\cS_k^{\textrm{new}}(\Gamma_0(3))^{\pm}$ and some period polynomial spaces.
\end{abstract}

\maketitle

\section{Introduction and main result}\label{sc1}
Recall that the multiple zeta values (MZVs) are defined by
$$\zeta(n_1,\ldots,n_r):=\sum_{0<k_1<\cdots<k_r}\frac{1}{k_1^{n_1}\cdots k_r^{n_r}}, \quad n_r\neq 1.$$
The weight is $n_1+\cdots+n_r$ and the depth is $r$. Brown considered in \cite{B} the motivic versions of those MZVs, denoted by $\zeta^\m$ which span the $\Q$-vector space of motivic multiple zeta values, denoted by $\cH$. The $\Q$-vector subspace generated by those with all $n_i$'s odd is called the totally odd motivic multiple zeta values, denoted by $\cH^{\odd}$. Both $\cH$ and $\cH^{\odd}$ are equipped with a depth filtration, denoted by $\depth$. He defined for every integer $n\geq 1$ a derivation on the depth-graded totally odd motivic multiple zeta values algebra
\begin{equation}\label{brownop}
\partial_{2n+1}:\gr_{r}^\depth\cH^{\odd}\to\gr_{r-1}^\depth\cH^{\odd}.
\end{equation}
If $m_1+\cdots+m_r=n_1\cdots+n_r$ are integers $\geq 1$, by taking $r$ consecutive derivations, Brown defined the following integers
\begin{equation}
c_{\left(\begin{smallmatrix}2m_1+1&\cdots&2m_r+1\\2n_1+1&\cdots&2n_r+1\end{smallmatrix}\right)}=\partial_{2m_r+1}\cdots\partial_{2m_1+1}\zeta^\m(2n_1+1,\ldots,2n_r+1)\in\Z.
\end{equation}
He claimed that a matrix consisting exactly those integers will capture all the linear relations among the totally odd motivic multiple zeta values of a given weight $k$ and depth $r$.\\
Let $\o\o(k)$ be the totally odd indexing set with the following order
$$\o\o(k)=\{(k-3,3),(k-5,5),\cdots,(3,k-3)\}.$$
Let us denote the reverse ordering indexing set to be $\o\o'(k)$, and
\begin{equation}
\cC_{k,2}^{1}=\big(c_{{x\choose y}}\big)_{\begin{subarray}{c} x \in \o\o(k) \\ y \in \o\o'(k)\end{subarray}}.
\end{equation}
In such a setting, Baumard and Schneps' result can be stated as follows.
\begin{theoremNoNum}[Baumard and Schneps \cite{BS}]
Let $k$ be an even integer. The left kernel of $\cC_{k,2}^{1}$ are exactly those vectors coming from the restricted even period polynomials of cusp forms of weight $k$ for $\SL_2(\Z)$.
\end{theoremNoNum}
Here we want to mention that the right kernel of $\cC_{k,2}^1$ gives us all the linear relations among totally odd double zeta values of weight $k$, which was completely studied by Gangl, Kaneko and Zagier in \cite{GKZ}.

In this paper, we will generalize Brown's matrix $\cC^{N}_{k,2}$ to level $N=2,3$ for double zeta values relative to $\mu_N$, and prove the following result.
\begin{theorem}[Connection with Hecke operators]\label{hecke}
Let $k$ be an even integer. When $N=2,3$, the vectors coming from the restricted even period polynomials of cuspidal eigenforms of weight $k$ for $\SL_2(\Z)$ are left eigenvectors of $\cC_{k,2}^{N}$, and the corresponding eigenvalues are given by
\begin{itemize}
\item $N=2$, 
$$\frac{\lambda_2-(1+2^{k-1})}{2^{k-2}},$$
\item $N=3$,
$$\frac{\lambda_3-(1+3^{k-1})}{4\cdot3^{k-2}},$$
\end{itemize}
where $\lambda_2$ (respectively, $\lambda_3$) is the eigenvalue for the Hecke operator $T_2$ (respectively, $T_3$) for the corresponding eigenform.
\end{theorem}

We will also prove the following result for the product $\cD_{k,2}^{N}\cdot\cC_{k,2}^{N}$, where $\cD_{k,2}^{N}$ is a natural diagonal matrix defined later.

\begin{theorem}[Connection with newforms]\label{newform}
Let $k$ be an even integer. When $N=2,3$, the vectors coming from the restricted even period polynomials of newforms of weight $k$ and level $\Gamma_0(N)$ are left eigenvectors of $(\cD_{k,2}^{N}\cdot\cC_{k,2}^{N})$ , and the corresponding eigenvalues are given by
\begin{itemize}
\item $N=2$, 
$$-\bigg(1+\frac{\varepsilon}{2^{\frac{k-2}{2}}}\bigg),$$
\item $N=3$,
$$-\frac{1}{2}\bigg(1+\frac{\varepsilon}{3^{\frac{k-2}{2}}}\bigg),$$
\end{itemize}
where the $\varepsilon=\pm1$ is the eigenvalues of the Atkin-Lehner involution of the corresponding newform.
\end{theorem}

\begin{remark}
It is worth mentioning that during the proofs of Theorem \ref{hecke} and Theorem \ref{newform}, we can see that the action of $\cC^{N}_{k,2}$ and $\cD^{N}_{k,2}\cdot\cC^{N}_{k,2}$ is nothing but the following two well-known operators (up to some scalar) on the corresponding spaces:
\begin{eqnarray}
\cC^{N}_{k,2} &\longleftrightarrow& T_N-1-N^{k-1}\qquad\textrm{acting on }r_f^{\ev,0}(x,y)\textrm{ of }f\in\cS_k(\SL_2(Z)),\\
\cD^{N}_{k,2}\cdot\cC^{N}_{k,2} &\longleftrightarrow& U_N-1\qquad\qquad\quad\ \textrm{acting on }r_f^{\ev,0}(x,y)\textrm{ of }f\in\cS_k^{\textrm{new}}(\Gamma_0(N))^{\pm},
\end{eqnarray}
where $r_f^{\ev,0}(x,y)$ is the restricted even period polynomial (introduced below) of $f$. We can see that those two theorems are compatible with Baumard and Schneps' result, since when $N=1$ both of them give us the zero action on the restricted even period polynomials of $f\in\cS_k(\SL_2(Z))$ (in this case, every cusp form is a new form).
\end{remark}

During the proof of Theorem \ref{newform}, we found two sets of equations seems to be held only for restricted even period polynomials of level $\Gamma_0(2)$ and $\Gamma_0(3)$ respectively. Therefore, it is natural to make the following conjectures.
\begin{conj}[Eichler-Shimura-Manin correspondence for $\cS_k^{\textrm{new}}(\Gamma_0(2))^{\pm}$ and $\cS_k^{\textrm{new}}(\Gamma_0(3))^{\pm}$]
We have the following isomorphism defined over $\C$.
\begin{eqnarray*}
\cS_k^{\mathrm{new}}(\Gamma_0(2))^{\pm}&\cong&W_{2,\mathrm{new},\pm}^{\ev,0}:=\bigg\{p(x,y)\in\C[x,y]\ \bigg|\ \begin{subarray}{c} -p(y,x)-p(y,x+y)+p(x,x+y)=-p(x,y)\\
-p(y,2x)=\pm 2^{\frac{k-2}{2}}p(x,y)\end{subarray}\bigg\},\\
\cS_k^{\mathrm{new}}(\Gamma_0(3))^{\pm}&\cong&W_{3,\mathrm{new},\pm}^{\ev,0}:=\bigg\{p(x,y)\in\C[x,y]\ \bigg|\ \begin{subarray}{c} -p(y,x)-p(y,x+y)+p(x,x+y)-p(y,x-y)+p(x,x-y)=-p(x,y)\\
-p(y,3x)=\pm 3^{\frac{k-2}{2}}p(x,y)\end{subarray}\bigg\}.
\end{eqnarray*}
\end{conj}

In Section \ref{sc2}, we will give the detailed background and also explicit construction of our matrices $\cC_{k,2}^N$ and $\cD_{k,2}^N$ for $N=2,3$. We will also list the results we need to use in our proof. In Section \ref{sc3}, we will give the explicit formula for $\cC_{k,2}^N$ and $\cD_{k,2}^N$. In Section \ref{sc5}, we will prove the formula for $U_N$-operator acting on the period polynomial of $f\in\cS_k(\Gamma_0(N))$ for $N=2,3$. In Section \ref{sc4} and Section \ref{sc6}, we will prove Theorem \ref{hecke} and Theorem \ref{newform} respectively. In Section \ref{sc7}, we will give some examples about Theorem \ref{hecke} and Theorem \ref{newform}. Finally, in Section \ref{sc8}, we will state some applications of those two theorems.

\section{Background, Definitions}\label{sc2}
In this section, we will introduce the notion of multiple zeta values relative to $\mu_N$. After introducing Glanois' result about derivation operators, we will define the matrix $\cC_{k,2}^N$ and $\cD_{k,2}^N$ for $N=2,3$.\\
Now let us look at the multiple zeta values relative to $\mu_N$. Recall that multiple zeta values relative to $\mu_N$ are defined by
$$\zeta\binom{n_1,\cdots ,n_r}{\varepsilon_1,\cdots,\varepsilon_r}:=\sum_{0<k_1<\cdots<k_r}\frac{\varepsilon_1^{k_1}\cdots \varepsilon_r^{k_r}}{k_1^{n_1}\cdots k_r^{n_r}},\quad \varepsilon_i\in\mu_N,\ (n_r,\varepsilon_r)\neq (1,1).$$
As in the MZVs case, the weight is $n_1+\cdots+n_r$, and the depth is $r$. Denote by $\cZ^N$ the $\Q$-vector space spanned by these multiple zeta values at arguments $n_i\in\Z_{\geq 1}$, $\varepsilon_i\in\mu_N$. Let us consider the motivic versions of those multiple zeta values, denoted by $\zeta^\m$ which span the $\Q$-vector space of motivic multiple zeta values relative to $\mu_N$, denoted by $\cH^N$. There is a surjective homomorphism called the period map as in the MZV case, conjectured also to be an isomorphism:
\begin{eqnarray}
per: \cH^N&\to&\cZ^N\\
\zeta^\m(\cdot)&\mapsto&\zeta(\cdot). \nonumber
\end{eqnarray}

Let $\Pi_{0,1}:=\pi_1^{\dR}(\P^1\setminus\{0,\mu_N,\infty\},\overrightarrow{1}_0,\overrightarrow{-1}_1)$ denote the de Rham realization of the motivic torsor of paths from $0$ to $1$ on $\P^1\setminus\{0,\mu_N,\infty\}$, with tangential basepoints given by the tangent vectors $1$ at $0$ and $-1$ at $1$. Its affine ring of regular functions is the graded algebra for the shuffle product:
\begin{equation}
\cO(\Pi_{0,1})\cong \Q\langle e^0,(e^\eta)_{\eta\in\mu_N}\rangle.
\end{equation}
Let $\cMT_N$ be the Tannakian category of the mixed Tate motives. Denote by $\cMT'_N$ the full Tannakian subcategory of $\cMT_N$ generated by the motivic fundamental groupoid of $\P^1\setminus\{0,\mu_N,\infty\}$. Denote also by $\cG_N=\G_m\ltimes\cU_N$ its Galois group defined over $\Q$, $\cA^N=\cO(\cU_N)$ its fundamental Hopf algebra and $\cL^N:=\cA^N_{>0}/\cA^N_{>0}\cdot \cA^N_{>0}$ the Lie coalgebra of indecomposable elements.

For any integer $N,p\geq 1$, Glanois defined the following derivation operators:
\begin{equation}
D_p:\cH^N\to\cL^N_r\otimes \cH^N
\end{equation}
by using the coaction $\Delta$ for motivic iterated integrals defined by Goncharov \cite{G1} and extended by Brown \cite{B1}. In \cite{Gl}, Glanois gave the following explicit formula for the operator $D_p$.

\begin{propNoNum}[Glanois \cite{Gl}]
For any integers $N,p,r\geq 1$, we have
\begin{eqnarray*}
D_{p}:\gr^\depth_r\cH^N&\to& \gr_1^\depth\cL^N_p\otimes\gr^\depth_{r-1}\cH^N\\
\zeta^\m\binom{n_1,\cdots ,n_r}{\varepsilon_1,\cdots,\varepsilon_r}&\mapsto&\delta_{p=n_1}\zeta^\l\binom{p}{\varepsilon_1}\otimes\zeta^\m\binom{n_2,\cdots}{\varepsilon_2,\cdots}\\ 
&+&\sum_{i=2}^{r-1}\delta_{n_i\leq p<n_i+n_{i-1}-1}(-1)^{p-n_i}\binom{p-1}{p-n_i}\zeta^\l\binom{p}{\varepsilon_i}\otimes\zeta^\m\binom{\cdots,n_i+n_{i-1}-p,\cdots}{\cdots,\varepsilon_{i-1}\varepsilon_i,\cdots}\\
&+&\sum_{i=1}^{r-1}\delta_{n_i\leq p<n_i+n_{i+1}-1}(-1)^{n_i}\binom{p-1}{p-n_i}\zeta^\l\binom{p}{\varepsilon_i^{-1}}\otimes\zeta^\m\binom{\cdots,n_i+n_{i+1}-p,\cdots}{\cdots,\varepsilon_{i+1}\varepsilon_i,\cdots}\\
&+&\sum_{i=2}^{r-1}\delta_{\substack{p=n_i+n_{i-1}-1\\ \varepsilon_{i-1}\varepsilon_i\neq 1}}\bigg((-1)^{n_i}\binom{p-1}{n_i-1}\zeta^\l\binom{p}{\varepsilon_{i-1}^{-1}}+(-1)^{n_{i-1}-1}\binom{p-1}{n_{i-1}-1}\zeta^\l\binom{p}{\varepsilon_i}\bigg)\\
&&\qquad\qquad\qquad\qquad\qquad\qquad\qquad\otimes\zeta^\m\binom{\cdots,1,\cdots}{\cdots,\varepsilon_{i-1}\varepsilon_i,\cdots}\\
&+&\delta_{n_r\leq p\leq n_{r}+n_{r-1}-1}(-1)^{p-n_r}\binom{p-1}{p-n_r}\zeta^\l\binom{p}{\varepsilon_r}\otimes\zeta^\m\binom{\cdots,n_{r-1}+n_r-p}{\cdots,\varepsilon_{r-1}\varepsilon_r}\\
&+&\delta_{\substack{p=n_{r}+n_{r-1}-1\\ \varepsilon_{r-1}\varepsilon_r\neq 1}}(-1)^{n_r-1}\bigg(\binom{p-1}{n_r-1}\zeta^\l\binom{p}{\varepsilon_{r-1}^{-1}}-\binom{p-1}{n_{r-1}-1}\zeta^\l\binom{p}{\varepsilon_r}\bigg)\\
&&\qquad\qquad\qquad\qquad\qquad\qquad\qquad\otimes\zeta^\m\binom{\cdots,1}{\cdots,\varepsilon_{r-1}\varepsilon_{r}}.
\end{eqnarray*}
\end{propNoNum}
\begin{remark}
There is a typo about one sign in \cite{Gl} for this proposition.
\end{remark}

For $N=2,3$, we have the following relations in depth $1$ due to Deligne and Goncharov \cite{DG}.
\begin{itemize}
\item For $N=2$:
\begin{equation}\label{depth12}
(2^{-2n}-1)\zeta^\l\binom{2n+1}{1}=\zeta^\l\binom{2n+1}{-1}.
\end{equation}
\item For $N=3$:
\begin{equation}\label{depth13}
\zeta^\l\binom{2n+1}{1}(1-3^{2n})=2\cdot 3^{2n}\zeta^\l\binom{2n+1}{\varepsilon_3},\quad \zeta^\l\binom{2n}{1}=0,\quad \zeta^\l\binom{n}{\varepsilon_3}=(-1)^{n-1}\zeta^\l\binom{n}{\varepsilon_3^{-1}},
\end{equation}
where $\varepsilon_3=e^{\frac{2\pi \sqrt{-1}}{3}}$.
\end{itemize}

For $N=2,3$, by using the above relations in depth $1$, Glanois \cite{Gl} defined the following analogous derivations as Brown did for MZVs. For each $(p,\eta)$, define
\begin{equation}
\partial^{\eta}_{p}:\gr_r^\depth\cH^N\to\gr_{r-1}^\depth\cH^N,
\end{equation}
as the composition of $D_p$ followed by the projection:
\begin{eqnarray*}
\pi^\eta:\gr_1^\depth\cL^N_p\otimes\gr_{r-1}^\depth\cH^N &\to& \gr_{r-1}^\depth\cH^N\\
\zeta^\l\binom{p}{\varepsilon}\otimes X&\mapsto& c_{\eta,\varepsilon,p}X,
\end{eqnarray*}
with $c_{\eta,\varepsilon,p}\in\Q$ the coefficients of $\zeta^\l\binom{p}{\eta}$ in the decomposition of $\zeta^\l\binom{p}{\varepsilon}$. Let us fixed the following basis of $\gr_1^\depth\cA_p^N$.
\begin{itemize}
\item For $N=2$, take $\zeta^\a\binom{p}{1}$ for odd $p\geq 1$.
\item For $N=3$, take $\zeta^\a\binom{p}{1}$ for odd $p\geq 1$, and $\zeta^\a\binom{p}{\omega}$ for even $p\geq 2$.
\end{itemize}
Now for this basis, let us simplify our notations about $\partial_p^\eta$ as follows:
\begin{itemize}
\item For $N=2$, write $\partial_p:=\partial_p^{1}$ for odd $p\geq 1$.
\item For $N=3$, write $\partial_p:=\partial_p^{1}$ for odd $p\geq 1$, and $\partial_p:=\partial_p^{\omega}$ for even $p\geq 2$.
\end{itemize}

Now we are ready to define the matrix $\cC_{k,2}^{N}$ and $\cD_{k,2}^{N}$. 
\begin{definition}[Definition of $\cC_{k,2}^{N}$]
Let $k$ be a positive even integer. For $N=2,3$, let $\varepsilon_N=e^{\frac{2\pi \sqrt{-1}}{N}}$. For any $(2m_1+1)+(2m_2+1)=(2n_1+1)+(2n_2+1)=k$, let us define
\begin{equation}
c^N_{\left(\begin{smallmatrix}2m_1+1&2m_2+1\\2n_1+1&2n_2+1\end{smallmatrix}\right)}=\partial_{2m_2+1}\partial_{2m_1+1}\zeta^\m\binom{2n_1+1,2n_2+1}{\varepsilon_N,\varepsilon_N^{-1}}\in\Q.
\end{equation}
We denote the matrix $\cC_{k,2}^{N}$ as follows
\begin{equation}
\cC_{k,2}^{N}=\big(c^N_{{x\choose y}}\big)_{\begin{subarray}{c} x \in \o\o(k) \\ y \in \o\o'(k)\end{subarray}}.
\end{equation}
\end{definition}

\begin{remark}
Notice that we have $\partial_1$ when $N=2,3$, and $\partial_{\even}$ when $N=3$, but here in the definition of $\cC_{k,2}^{N}$ we still only consider $\partial_p$ with odd $p\geq 3$ as in the $\cC_{k,2}^{1}$ case, since we want to set up some connections with restricted even period polynomials.
\end{remark}

\begin{definition}[Definition of $\cD_{k,2}^{N}$]
For any positive even integer $k$, and $N=2,3$, let us define $\cD_{k,2}^{N}$ to be the following diagonal matrix
\begin{equation}\label{cD}
\cD_{k,2}^{N}=\bigg(\delta\binom{m_1,m_2}{n_1,n_2}\cdot c_{1,\varepsilon_N,m_1}^{-1}\bigg)_{\begin{subarray}{c} (m_1,m_2) \in \o\o(k) \\ (n_1,n_2) \in \o\o(k)\end{subarray}},
\end{equation}
where $\varepsilon_N=e^{\frac{2\pi \sqrt{-1}}{N}}$ and $c_{1,\varepsilon_N,p}\in\Q$ is the coefficients of $\zeta^\l\binom{p}{1}$ in the decomposition of $\zeta^\l\binom{p}{\varepsilon_N}$ given in (\ref{depth12}) and (\ref{depth13}).
\end{definition}

\section{Explicit Formulas of $\cC_{k,2}^{N}$ and $\cD_{k,2}^{N}\cdot\cC_{k,2}^{N}$ }\label{sc3}
In this section, we will prove explicit formulas for both $\cC_{k,2}^{N}$ and $\cD_{k,2}^{N}\cdot\cC_{k,2}^{N}$.
\begin{lemma}\label{lem1}
For $N=2,3$, let $k$ be a positive even integer and $(m_1,m_2)\in\o\o(k)$ and $(n_1,n_2)\in\o\o'(k)$, we have
\begin{equation*}
(\cC_{k,2}^N)_{\binom{m_1,m_2}{n_1,n_2}}=c_{1,\varepsilon_N,m_1}\bigg(\delta\binom{m_1,m_2}{n_1,n_2}c_{1,\varepsilon_N,m_2}+(-1)^{n_1}\binom{m_1-1}{m_1-n_1}+(-1)^{m_1-n_2}\binom{m_1-1}{m_1-n_2}\bigg),
\end{equation*}
where
\begin{equation*}
c_{1,1,p}=1,\quad c_{1,-1,p}=(2^{-p+1}-1),\quad c_{1,\varepsilon_3,p}=c_{1,\varepsilon_3^{-1},p}=\frac{3^{-p+1}-1}{2}
\end{equation*}
as in (\ref{depth12}) and (\ref{depth13}).
\end{lemma}

\begin{proof}
For $N=2,3$, and odd $p\geq 3$, let us first compute $D_p$ and $\partial_p$ explicitly on the elements $\zeta^\m\binom{n}{\varepsilon_1}$ for any $N$th root $\varepsilon_1$ and $\zeta^\m\binom{n_1,n_2}{\varepsilon_N,\varepsilon_N^{-1}}$ for $\varepsilon_N=e^{\frac{2\pi \sqrt{-1}}{N}}$ by using Glanois' result, (\ref{depth12}), and (\ref{depth13}).\\
When $r=1$, we have
\begin{equation}\label{dp1}
D_p:\zeta^\m\binom{n}{\varepsilon_1}\longmapsto\delta_{p=n}\zeta^\l\binom{p}{\varepsilon_1}.
\end{equation}
When $r=2$, we have 
\begin{eqnarray*}
D_p:\zeta^\m\binom{n_1,n_2}{\varepsilon_N,\varepsilon_N^{-1}}&\longmapsto&\delta_{p=n_1}\zeta^\l\binom{p}{\varepsilon_N}\otimes\zeta^\m\binom{n_2}{\varepsilon_N^{-1}}\\
&+&\delta_{n_1\leq p< n_1+n_2-1}(-1)^{n_1}\binom{p-1}{p-n_1}\zeta^\l\binom{p}{\varepsilon_N^{-1}}\otimes\zeta^\m\binom{n_1+n_2-p}{1}\\
&+&\delta_{n_2\leq p\leq n_{1}+n_{2}-1}(-1)^{p-n_2}\binom{p-1}{p-n_2}\zeta^\l\binom{p}{\varepsilon_N^{-1}}\otimes\zeta^\m\binom{n_{1}+n_2-p}{1}
\end{eqnarray*}
If we use the convention that $\binom{r}{s}=0$ if $s<0<r$ and $\zeta^\m\binom{n}{1}=0$ if $n\leq 0$, the above expression can be simplified as
\begin{eqnarray}\label{dp2}
D_p:\zeta^\m\binom{n_1,n_2}{\varepsilon_N,\varepsilon_N^{-1}}&\longmapsto&\delta_{p=n_1}\zeta^\l\binom{p}{\varepsilon_N}\otimes\zeta^\m\binom{n_2}{\varepsilon_N^{-1}}\\
&+&(-1)^{n_1}\binom{p-1}{p-n_1}\zeta^\l\binom{p}{\varepsilon_N^{-1}}\otimes\zeta^\m\binom{n_1+n_2-p}{1} \nonumber\\
&+&(-1)^{p-n_2}\binom{p-1}{p-n_2}\zeta^\l\binom{p}{\varepsilon_N^{-1}}\otimes\zeta^\m\binom{n_{1}+n_2-p}{1}\nonumber
\end{eqnarray}
Now from (\ref{dp1}) and (\ref{dp2}), when $(m_1,m_2)\in\o\o(k)$ and $(n_1,n_2)\in\o\o'(k)$, we have
\begin{eqnarray*}
(\cC_{k,2}^N)_{\binom{m_1,m_2}{n_1,n_2}}&=&\partial_{m_2}\partial_{m_1}\zeta^\m\binom{n_1,n_2}{\epsilon_N,\epsilon_N^{-1}}\\
&=&c_{1,\varepsilon_N,m_1}\partial_{m_2}\bigg(\delta_{m_1=n_1}\zeta^\m\binom{n_2}{\varepsilon_N^{-1}}+(-1)^{n_1}\binom{m_1-1}{m_1-n_1}\zeta^\m\binom{n_1+n_2-m_1}{1}\\
&&\qquad\qquad\qquad\qquad +(-1)^{m_1-n_2}\binom{m_1-1}{m_1-n_2}\zeta^\m\binom{n_{1}+n_2-m_1}{1}\bigg)\\
&=&c_{1,\varepsilon_N,m_1}\bigg(\delta\binom{m_1,m_2}{n_1,n_2}c_{1,\varepsilon_N,m_2}+(-1)^{n_1}\binom{m_1-1}{m_1-n_1}+(-1)^{m_1-n_2}\binom{m_1-1}{m_1-n_2}\bigg)
\end{eqnarray*}
\end{proof}

From the above lemma and (\ref{cD}), we have the following explicit formula for $\cD_{k,2}^{N}\cdot\cC_{k,2}^{N}$
\begin{lemma}\label{lem2}
For $N=2,3$, let $k$ be a positive even integer and $(m_1,m_2)\in\o\o(k)$ and $(n_1,n_2)\in\o\o'(k)$, we have
\begin{equation*}
(\cD_{k,2}^N\cdot\cC_{k,2}^N)_{\binom{m_1,m_2}{n_1,n_2}}=\delta\binom{m_1,m_2}{n_1,n_2}c_{1,\varepsilon_N,m_2}+(-1)^{n_1}\binom{m_1-1}{m_1-n_1}+(-1)^{m_1-n_2}\binom{m_1-1}{m_1-n_2}.
\end{equation*}
\end{lemma}

\section{Proof of Theorem \ref{hecke}}\label{sc4}
In this section, we will give the proof of theorem. First, we want to introduce the representative sets of the Hecke operator $T_2$, $T_3$ acting on the period polynomial spaces. The following result is due to Zagier.
\begin{propNoNum}[Zagier \cite{Z}]
The representative sets of the the Hecke operator $T_2$, $T_3$ acting on the even period polynomial spaces are given by
\begin{eqnarray*}
\textbf{Man}_2&=&\bigg\{\begin{pmatrix}2&0\\0&1\end{pmatrix},\begin{pmatrix}1&0\\0&2\end{pmatrix},\begin{pmatrix}2&0\\1&1\end{pmatrix},\begin{pmatrix}1&1\\0&2\end{pmatrix}\bigg\},\\
\textbf{Man}_3&=&\bigg\{\begin{pmatrix}3&0\\0&1\end{pmatrix},\begin{pmatrix}1&0\\0&3\end{pmatrix},\begin{pmatrix}3&0\\1&1\end{pmatrix},\begin{pmatrix}3&0\\-1&1\end{pmatrix},\begin{pmatrix}1&1\\0&3\end{pmatrix},\begin{pmatrix}1&-1\\0&3\end{pmatrix}\bigg\}.
\end{eqnarray*}
\end{propNoNum}
Now we are ready to prove Theorem \ref{hecke}.

\begin{proof}[Proof of Theorem \ref{hecke}]
When $N=2$, for any $(m_1,m_2)\in\o\o(k)$ and $(n_1,n_2)\in\o\o'(k)$, let us define for simplicity that
\begin{equation*}
e\binom{m_1,m_2}{n_1,n_2}=(-1)^{n_1}\binom{m_1-1}{m_1-n_1}+(-1)^{m_1-n_2}\binom{m_1-1}{m_1-n_2}.
\end{equation*}
Notice that this $e$ is exactly the main term in the definition of $(\cC_{k,2}^1)_{\binom{m_1,m_2}{n_1,n_2}}$, i.e.
\begin{equation*}
(\cC_{k,2}^1)_{\binom{m_1,m_2}{n_1,n_2}}=\delta\binom{m_1,m_2}{n_1,n_2}+e\binom{m_1,m_2}{n_1,n_2}.
\end{equation*}

From Lemma \ref{lem1}, we have
\begin{eqnarray*}
(\cC_{k,2}^2)_{\binom{m_1,m_2}{n_1,n_2}}&=&c_{1,-1,m_1}\bigg(\delta\binom{m_1,m_2}{n_1,n_2}c_{1,-1,m_2}+e\binom{m_1,m_2}{n_1,n_2}\bigg)\\
&=&(2^{-m_1+1}-1)(2^{-m_2+1}-1)\delta\binom{m_1,m_2}{n_1,n_2}+(2^{-m_1+1}-1)e\binom{m_1,m_2}{n_1,n_2}\\
&=&-2^{-m_1+1}\delta\binom{m_1,m_2}{n_1,n_2}-2^{-m_2+1}\delta\binom{m_1,m_2}{n_1,n_2}+2^{-m_1+1}e\binom{m_1,m_2}{n_1,n_2}\\
&&-(\cC_{k,2}^1)_{\binom{m_1,m_2}{n_1,n_2}}+2\delta\binom{m_1,m_2}{n_1,n_2}+2^{-k+2}\delta\binom{m_1,m_2}{n_1,n_2}
\end{eqnarray*}
Define
\begin{eqnarray*}
A_{\binom{m_1,m_2}{n_1,n_2}}:=2^{-m_1+1}\delta\binom{m_1,m_2}{n_1,n_2},\\
B_{\binom{m_1,m_2}{n_1,n_2}}:=2^{-m_2+1}\delta\binom{m_1,m_2}{n_1,n_2},\\
C_{\binom{m_1,m_2}{n_1,n_2}}:=2^{-m_1+1}e\binom{m_1,m_2}{n_1,n_2}.
\end{eqnarray*}
Then we have the following decomposition
\begin{equation*}
\cC_{k,2}^2=-A-B+C-\cC_{k,2}^1+2J+2^{-k+2}J,
\end{equation*}
where $J$ is the exchange matrix, i.e. anti-diagonal matrix with all $1$'s.
For any eigenform $f$ of weight $k$ and level $\SL_2(\Z)$ with period polynomial 
\begin{equation}
r_f(x,y)=\int_0^{i\infty}f(z)(zy-x)^{k-2}dz.
\end{equation}
Consider the restricted even part of this period polynomial, we have
\begin{equation}
r^{\ev,0}_f(x,y)=\sum_{(m_1,m_2)\in\o\o(k)} a_{m_1,m_2}x^{m_2-1}y^{m_1-1}.
\end{equation}
We associate this restricted even period polynomial with the following vector
\begin{equation}
v=(a_{m_1,m_2})_{(m_1,m_2)\in\o\o(k)}.
\end{equation}
By the definition of the matrices $A,B,C,J$, we have
\begin{eqnarray*}
\cdot A&:& p(x,y)\mapsto p(\frac{y}{2},x),\\
\cdot B&:& p(x,y)\mapsto p(y,\frac{x}{2}),\\
\cdot C&:& p(x,y)\mapsto p(x,\frac{x+y}{2})-p(y,\frac{x+y}{2}),\\
\cdot J&:& p(x,y)\mapsto p(y,x).
\end{eqnarray*}
Also from Baumard and Schneps' result, we know that
\begin{equation*}
\cdot \cC_{k,2}^1: r^{\ev,0}_f(x,y)\mapsto 0.
\end{equation*}
Therefore, we have
\begin{eqnarray*}
\cdot \cC_{k,2}^2: r^{\ev,0}_f(x,y)&\mapsto& -r^{\ev,0}_f(\frac{y}{2},x)-r^{\ev,0}_f(y,\frac{x}{2})+r^{\ev,0}_f(x,\frac{x+y}{2})-r^{\ev,0}_f(y,\frac{x+y}{2})+(2+2^{-k+2})r^{\ev,0}_f(y,x)\\
&=&r^{\ev,0}_f(x,\frac{y}{2})+r^{\ev,0}_f(\frac{x}{2},y)+r^{\ev,0}_f(x,\frac{x+y}{2})+r^{\ev,0}_f(\frac{x+y}{2},y)-(2+2^{-k+2})r^{\ev,0}_f(x,y)\\
&=&2^{-k+2}\bigg(r^{\ev,0}_f(2x,y)+r^{\ev,0}_f(x,2y)+r^{\ev,0}_f(2x,x+y)+r^{\ev,0}_f(x+y,2y)\bigg)\\
&-&(2+2^{-k+2})r^{\ev,0}_f(x,y)
\end{eqnarray*}
From Zagier's result above, if $f$ is an eigenform with $T_2$-eigenvalues $\lambda_2$, we have
\begin{equation*}
r^{\ev,0}_f(2x,y)+r^{\ev,0}_f(x,2y)+r^{\ev,0}_f(2x,x+y)+r^{\ev,0}_f(x+y,2y)=r^{\ev,0}_{f|T_2}(x,y)=\lambda_2 r^{\ev,0}_f(x,y).
\end{equation*}
Therefore,
\begin{eqnarray*}
\cdot \cC_{k,2}^2: r^{\ev,0}_f(x,y)&\mapsto&2^{-k+2}\bigg(r^{\ev,0}_f(2x,y)+r^{\ev,0}_f(x,2y)+r^{\ev,0}_f(2x,x+y)+r^{\ev,0}_f(x+y,2y)\bigg)\\
&-&(2+2^{-k+2})r^{\ev,0}_f(x,y)\\
&=&\frac{\lambda_2-(1+2^{k-2})}{2^{k-2}}r^{\ev,0}_f(x,y).
\end{eqnarray*}
Hence we have proven the statement for $N=2$.\\
Now let us prove it for $N=3$. By the same computation, we have
\begin{eqnarray*}
\cdot \cC_{k,2}^3: r^{\ev,0}_f(x,y)&\mapsto&\frac{1}{4}r^{\ev,0}_f(x,\frac{y}{3})+\frac{1}{4}r^{\ev,0}_f(\frac{x}{3},y)+\frac{1}{2}r^{\ev,0}_f(x,\frac{x+y}{3})+\frac{1}{2}r^{\ev,0}_f(\frac{x+y}{3},y)\\
&-&\frac{1}{4}(3+3^{-k+2})r^{\ev,0}_f(x,y)\\
&=&\frac{3^{-k+2}}{4}\bigg(r^{\ev,0}_f(3x,y)+r^{\ev,0}_f(x,3y)+2r^{\ev,0}_f(3x,x+y)+2r^{\ev,0}_f(x+y,3y)\bigg)\\
&-&\frac{3+3^{-k+2}}{4}r^{\ev,0}_f(x,y),
\end{eqnarray*}
where $\frac{1}{2}$ and $\frac{1}{4}$ are coming from the $\frac{1}{2}$ in the following identity
\begin{equation*}
c_{1,\varepsilon_3,p}=c_{1,\varepsilon_3^{-1},p}=\frac{3^{-p+1}-1}{2}.
\end{equation*}
Now since $r^{\ev,0}_f(x,y)$ is an even polynomial, and we only need the even degree part in the final expression, we have
\begin{eqnarray*}
2r^{\ev,0}_f(3x,x+y)&\equiv& r^{\ev,0}_f(3x,x+y)+r^{\ev,0}_f(3x,x-y)\quad\bmod(x^{\odd}y^{\odd}),\\
2r^{\ev,0}_f(x+y,xy)&\equiv& r^{\ev,0}_f(x+y,3y)+r^{\ev,0}_f(x-y,3y)\quad\bmod(x^{\odd}y^{\odd}).\\
\end{eqnarray*}
Therefore, we have
\begin{eqnarray*}
\cdot \cC_{k,2}^3: p(x,y)&\mapsto&\frac{3^{-k+2}}{4}\bigg(r^{\ev,0}_f(3x,y)+r^{\ev,0}_f(x,3y)+2r^{\ev,0}_f(3x,x+y)+2r^{\ev,0}_f(x+y,3y)\bigg)\\
&-&\frac{3+3^{-k+2}}{4}r^{\ev,0}_f(x,y)\\
&=&\frac{3^{-k+2}}{4}r^{\ev,0}_{f|T_3}(x,y)-\frac{3+3^{-k+2}}{4}r^{\ev,0}_f(x,y)\\
&=&\frac{\lambda_3-(1+3^{k-1})}{4\cdot3^{k-2}}r^{\ev,0}_f(x,y).
\end{eqnarray*}
Hence we have also proven the statement for $N=3$.
\end{proof}

\section{$U_2$, $U_3$-action on period polynomials}\label{sc5}
In this section, we will prove a result about the $U_N$-operator acting on the period polynomial of cusp form of weight $k$ and level $\Gamma_0(N)$ for $N=2,3$.
\begin{lemma}
For any $f\in\cS_k(\Gamma_0(2))$, let $r_f(x,y)$ denote its period polynomial, then we have
\begin{equation}\label{U2}
r_{f|U_2}(x,y)=r_f(x,2y)+r_f(x+y,2y)-r_f(x+y,-2x)
\end{equation}
\end{lemma}
\begin{proof}
By definition, we have
\begin{equation*}
r_f(x,y)=\int_0^{i\infty}f(z)(zy-x)^{k-2}dz.
\end{equation*}
For $f\in\cS_k(\Gamma_0(2))$, the coset representative for the $U_2$-operator is given by $(\begin{smallmatrix}1&0\\0&2\end{smallmatrix})$ and $(\begin{smallmatrix}1&1\\0&2\end{smallmatrix})$. Then we have
\begin{eqnarray*}
r_{f|U_2}(x,y)&=&\int_0^{i\infty}\frac{1}{2}f\bigg(\frac{z}{2}\bigg)(zy-x)^{k-2}dz+\int_0^{i\infty}\frac{1}{2}f\bigg(\frac{z+1}{2}\bigg)(zy-x)^{k-2}dz\\
&=&\int_0^{i\infty}f(z)(2zy-x)^{k-2}dz+\int_{\frac{1}{2}}^{i\infty}f(z)((2z-1)y-x)^{k-2}dz\\
&=&\int_0^{i\infty}f(z)(2yz-x)^{k-2}dz+\bigg(\int_{0}^{i\infty}-\int_{0}^{\frac{1}{2}}\bigg)f(z)(2yz-(x+y))^{k-2}dz\\
&=&r_f(x,2y)+r_f(x+y,2y)-\int_{0}^{\frac{1}{2}}f(z)(2yz-(x+y))^{k-2}dz\\
&=&r_f(x,2y)+r_f(x+y,2y)-\int_{0}^{i\infty}f(\gamma z)(2y\gamma(z)-(x+y))^{k-2}d\gamma(z),
\ \gamma=(\begin{smallmatrix}1&0\\2&1\end{smallmatrix})\in\Gamma_0(2)\\
&=&r_f(x,2y)+r_f(x+y,2y)-\int_{0}^{i\infty}(2z+1)^k f|[\gamma]_k(z)\frac{1}{(2z+1)^{k}}(-2xz-(x+y))^{k-2}dz\\
&=&r_f(x,2y)+r_f(x+y,2y)-\int_{0}^{i\infty}f(z)(-2xz-(x+y))^{k-2}dz\\
&=&r_f(x,2y)+r_f(x+y,2y)-r_f(x+y,-2x).
\end{eqnarray*}
\end{proof}

\begin{lemma}
For any $f\in\cS_k(\Gamma_0(3))$, let $r_f(x,y)$ denote its period polynomial, then we have
\begin{equation}\label{U3}
r_{f|U_3}(x,y)=r_f(x,3y)+r_f(x+y,3y)-r_f(x+y,-3x)+r_f(x-y,3y)-r_f(-(x-y),-3x)
\end{equation}
\end{lemma}
\begin{proof}
By definition, we have
\begin{equation*}
r_f(x,y)=\int_0^{i\infty}f(z)(zy-x)^{k-2}dz.
\end{equation*}
For $f\in\cS_k(\Gamma_0(3))$, the coset representative for the $U_3$-operator is given by $(\begin{smallmatrix}1&0\\0&3\end{smallmatrix})$, $(\begin{smallmatrix}1&1\\0&3\end{smallmatrix})$, and $(\begin{smallmatrix}1&2\\0&3\end{smallmatrix})$. Then we have
\begin{eqnarray*}
&&r_{f|U_3}(x,y)\\
&=&\int_0^{i\infty}\frac{1}{3}f\bigg(\frac{z}{3}\bigg)(zy-x)^{k-2}dz+\int_0^{i\infty}\frac{1}{3}f\bigg(\frac{z+1}{3}\bigg)(zy-x)^{k-2}dz+\int_0^{i\infty}\frac{1}{3}f\bigg(\frac{z+2}{3}\bigg)(zy-x)^{k-2}dz\\
&=&\int_0^{i\infty}f(z)(3zy-x)^{k-2}dz+\int_{\frac{1}{3}}^{i\infty}f(z)((3z-1)y-x)^{k-2}dz+\int_{\frac{2}{3}}^{i\infty}f(z)((3z-2)y-x)^{k-2}dz\\
&=&\int_0^{i\infty}f(z)(3yz-x)^{k-2}dz+\bigg(\int_{0}^{i\infty}-\int_{0}^{\frac{1}{3}}\bigg)f(z)(3yz-(x+y))^{k-2}dz\\
&&\qquad\qquad+\bigg(\int_{1}^{i\infty}-\int_{1}^{\frac{2}{3}}\bigg)f(z)(3yz-(x+2y))^{k-2}dz\\
&=&r_f(x,3y)+r_f(x+y,3y)-\int_{0}^{\frac{1}{3}}f(z)(3yz-(x+y))^{k-2}dz\\
&&\qquad\qquad+\int_{1}^{i\infty}f(z)(3yz-(x+2y))^{k-2}dz-\int_{1}^{\frac{2}{3}}f(z)(3yz-(x+2y))^{k-2}dz.
\end{eqnarray*}
Now let us compute the last three integrals.
\begin{eqnarray*}
\int_{0}^{\frac{1}{3}}f(z)(3yz-(x+y))^{k-2}dz&=&\int_{0}^{i\infty}f(\gamma_1 z)(3y\gamma_1(z)-(x+y))^{k-2}d\gamma_1(z),\ \gamma_1=(\begin{smallmatrix}1&0\\3&1\end{smallmatrix})\in\Gamma_0(3)\\
&=&\int_{0}^{i\infty}f(z)(-3xz-(x+y))^{k-2}dz\\
&=&r_f(x+y,-3x)
\end{eqnarray*}
\begin{eqnarray*}
\int_{1}^{i\infty}f(z)(3yz-(x+2y))^{k-2}dz&=&\int_{0}^{i\infty}f(\gamma_2 z)(3y\gamma_2(z)-(x+2y))^{k-2}d\gamma_2(z),\ \gamma_2=(\begin{smallmatrix}1&1\\0&1\end{smallmatrix})\in\Gamma_0(3)\\
&=&\int_{0}^{i\infty}f(z)(3yz-(x-y))^{k-2}dz\\
&=&r_f(x-y,3y)
\end{eqnarray*}
\begin{eqnarray*}
\int_{1}^{\frac{2}{3}}f(z)(3yz-(x+2y))^{k-2}dz&=&\int_{0}^{i\infty}f(\gamma_3 z)(3y\gamma_3(z)-(x+2y))^{k-2}d\gamma_3(z),\ \gamma_3=(\begin{smallmatrix}2&-1\\3&-1\end{smallmatrix})\in\Gamma_0(3)\\
&=&\int_{0}^{i\infty}f(z)(-3xz+(x-y))^{k-2}dz\\
&=&r_f(-(x-y),-3x)
\end{eqnarray*}
Therefore, we have
\begin{eqnarray*}
&&r_{f|U_3}(x,y)\\
&=&r_f(x,3y)+r_f(x+y,3y)-\int_{0}^{\frac{1}{3}}f(z)(3yz-(x+y))^{k-2}dz\\
&&\qquad\qquad+\int_{1}^{i\infty}f(z)(3yz-(x+2y))^{k-2}dz-\int_{1}^{\frac{2}{3}}f(z)(3yz-(x+2y))^{k-2}dz\\
&=&r_f(x,3y)+r_f(x+y,3y)-r_f(x+y,-3x)+r_f(x-y,3y)-r_f(-(x-y),-3x)
\end{eqnarray*}
\end{proof}
In particular, (\ref{U2}) and (\ref{U3}) also hold for restricted even period polynomials. We also need the following fact about the action of the Atkin-Lehner involution $W_N$ defined by
\begin{eqnarray*}
W_N:\cS_k(\Gamma_1(N))&\to&\cS_k(\Gamma_1(N))\\
f&\mapsto& N^{\frac{2-k}{2}}f|(\begin{smallmatrix}0&-1\\N&0\end{smallmatrix})_k
\end{eqnarray*}
For such a operator $W_N$, we can decompose $\cS_k(\Gamma_1(N))$ into two eigenspaces $\cS_k(\Gamma_1(N))^{\pm}$ such that
\begin{equation*}
\cS_k(\Gamma_1(N))^{\pm}=\{f\in \cS_k(\Gamma_1(N))\ |\ W_Nf=\pm f\}.
\end{equation*}
Remember that when $k$ is even and $N=2,3$, we have $\cS_k(\Gamma_1(N))=\cS_k(\Gamma_0(N))$. Therefore in those cases, this operator $W_N$ also gives us a decomposition
\begin{equation*}
\cS_k(\Gamma_0(N))=\cS_k(\Gamma_0(N))^{+}\oplus \cS_k(\Gamma_0(N))^{-}.
\end{equation*}
For each $f\in\cS_k(\Gamma_0(N))^{\pm}$, denote its eigenvalues by $\varepsilon_f$, then for the corresponding period polynomial, we have
\begin{eqnarray*}
&&N^{\frac{k-2}{2}}\varepsilon_f\cdot r_f(x,y)\\
&=&r_{f|(\begin{smallmatrix}0&-1\\N&0\end{smallmatrix})_k}(x,y)\\
&=&\int_0^{i\infty}N^{k-1}\frac{1}{(Nz)^k}f(\frac{-1}{Nz})(zy-x)^{k-2}dz\\
\bigg(z\mapsto \frac{-1}{Nz}\bigg)\qquad&=&\int_{i\infty}^0 N^{k-1}(-z)^k f(z)\bigg(\frac{-y}{Nz}-x\bigg)^{k-2}d\bigg(\frac{-1}{Nz}\bigg)\\
&=&-\int_0^{i\infty}f(z)(xNz+y)^{k-2}dz\\
&=&-r_f(-y,Nx).
\end{eqnarray*}
Therefore, for the restricted even period polynomial of $f\in\cS_k(\Gamma_0(N))^{\pm}$, we have
\begin{equation}\label{WN}
r_f^{\ev,0}(y,Nx)=-N^{\frac{k-2}{2}}\varepsilon_f\cdot r_f^{\ev,0}(x,y)
\end{equation}
 
\section{Proof of Theorem \ref{newform}}\label{sc6}
Now we are ready to prove Theorem \ref{newform}.
\begin{proof}[Proof of Theorem \ref{newform}]
When $N=2$, as in the proof of Theorem \ref{hecke}, we know that for any restricted even period polynomial $r_f^{\ev,0}(X,Y)$ of $f\in \cS_k^{\textrm{new}}(\Gamma_0(2))^{\pm}$, the matrix $(\cD_{k,2}^2\cdot \cC_{k,2}^2)$ maps it to
\begin{eqnarray*}
\cdot (\cD_{k,2}^2\cdot \cC_{k,2}^2): r_f^{\ev,0}(x,y)&\mapsto&-r_f^{\ev,0}(y,x)-r_f^{\ev,0}(y,x+y)+r_f^{\ev,0}(x,x+y)+r_f^{\ev,0}\bigg(\frac{y}{2},x\bigg)\\
&=&\frac{\varepsilon_f}{2^{\frac{k-2}{2}}}\bigg(r_f^{\ev,0}(x,2y)+r_f^{\ev,0}(x+y,2y)-r_f^{\ev,0}(x+y,2x)-r_f^{\ev,0}(x,y)\bigg) \\
&=&\frac{\varepsilon_f}{2^{\frac{k-2}{2}}}\cdot\bigg(r_{f|U_2}^{\ev,0}(x,y)-r_f^{\ev,0}(x,y)\bigg)\\
&=&\frac{\varepsilon_f}{2^{\frac{k-2}{2}}}\cdot\bigg(-2^{\frac{k-2}{2}}\varepsilon_f\cdot r_f^{\ev,0}(x,y)-r_f^{\ev,0}(x,y)\bigg)\\
&=&\bigg(-1-\frac{\varepsilon_f}{2^{\frac{k-2}{2}}}\bigg)r_f^{\ev,0}(x,y),
\end{eqnarray*}
where the first two lines are considered modulo $(x^{k-2}-y^{k-2})$. Notice that in the first line, we already know it is even period polynomial by the formula of $U_2$-action, so we do not need to cancel the odd degree term. Hence we have shown the statement for $N=2$.\\
When $N=3$, for any restricted even period polynomial $r_f^{\ev,0}(X,Y)$ of $f\in \cS_k^{\textrm{new}}(\Gamma_0(3))^{\pm}$, we have
\begin{eqnarray*}
\cdot (\cD_{k,2}^3\cdot \cC_{k,2}^3): r_f^{\ev,0}(x,y)&\mapsto&\frac{1}{2}\bigg(-\frac{1}{2}r_f^{\ev,0}(y,x)-r_f^{\ev,0}(y,x+y)+r_f^{\ev,0}(x,x+y)\bigg)\\
&+&\frac{1}{2}\bigg(-\frac{1}{2}r_f^{\ev,0}(-y,x)-r_f^{\ev,0}(-y,x-y)+r_f^{\ev,0}(x,x-y)\bigg)+\frac{1}{2}r_f^{\ev,0}\bigg(\frac{y}{3},x\bigg)\\
&=&\frac{1}{2}\bigg(-r_f^{\ev,0}(y,x)-r_f^{\ev,0}(y,x+y)+r_f^{\ev,0}(x,x+y)\\
&&\qquad\qquad-r_f^{\ev,0}(y,x-y)+r_f^{\ev,0}(x,x-y)\bigg)+\frac{1}{2}r_f^{\ev,0}\bigg(\frac{y}{3},x\bigg)\\
&=&\frac{\varepsilon_f}{2\cdot 3^{\frac{k-2}{2}}}\bigg(r_f^{\ev,0}(x,3y)+r_f^{\ev,0}(x+y,3y)-r_f^{\ev,0}(x+y,3x)\\
&&\qquad\qquad+r_f^{\ev,0}(x-y,3y)-r_f^{\ev,0}(x-y,3x)-r_f^{\ev,0}(x,y)\bigg)\\
&=&\frac{\varepsilon_f}{2\cdot3^{\frac{k-2}{2}}}\cdot\bigg(r_{f|U_3}^{\ev,0}(x,y)-r_f^{\ev,0}(x,y)\bigg)\\
&=&\frac{\varepsilon_f}{2\cdot3^{\frac{k-2}{2}}}\cdot\bigg(-3^{\frac{k-2}{2}}\varepsilon_f\cdot r_f^{\ev,0}(x,y)-r_f^{\ev,0}(x,y)\bigg)\\
&=&\bigg(-\frac{1}{2}-\frac{\varepsilon_f}{2\cdot 3^{\frac{k-2}{2}}}\bigg)r_f^{\ev,0}(x,y),
\end{eqnarray*}
where the first two lines are considered modulo $(x^{k-2}-y^{k-2})$. The first line here split into to parts since we want to cancel the odd degree terms. Hence we have also shown the statement for $N=3$.
\end{proof}

\section{Examples}\label{sc7}
In this section, we will provide some example for Theorem \ref{hecke} and Theorem \ref{newform}.
\begin{example}
When $k=12$, we have
\begin{eqnarray*}
\cC_{12,2}^{2}&=&
\left(\begin{array}{rrrr}
\frac{6885}{256} & \frac{5355}{128} & -\frac{5355}{128} & -\frac{26775}{1024} \\
\frac{945}{64} & \frac{441}{32} & -\frac{13167}{1024} & -\frac{945}{64} \\
\frac{45}{8} & \frac{1905}{1024} & -\frac{15}{16} & -\frac{45}{8} \\
\frac{1533}{1024} & 0 & 0 & -\frac{3}{4}
\end{array}\right)\\
\cC_{12,2}^{3}&=&\left(\begin{array}{rrrr}
\frac{3280}{243} & \frac{45920}{2187} & -\frac{45920}{2187} & -\frac{783920}{59049} \\
\frac{1820}{243} & \frac{5096}{729} & -\frac{398216}{59049} & -\frac{1820}{243} \\
\frac{80}{27} & \frac{43720}{59049} & -\frac{40}{81} & -\frac{80}{27} \\
\frac{39364}{59049} & 0 & 0 & -\frac{4}{9}
\end{array}\right)
\end{eqnarray*}
Those two matrices both has one left eigenvector $(1,-3,3,-1)$, and it corresponds to the restricted even period polynomial
\begin{equation*}
r_f^{\ev,0}(x,y)=x^2y^8-3x^4y^6+3x^6y^4-x^8y^2
\end{equation*}
of the unique cusp form in $\cS_{12}(\SL_2(\Z))$. The corresponding eigenvalues are
\begin{itemize}
\item $N=2$, 
$$\frac{\lambda_2-(1+2^{k-1})}{2^{k-2}}=-\frac{2073}{1024}=691\cdot\frac{-3}{1024}.$$
\item $N=3$,
$$\frac{\lambda_3-(1+3^{k-1})}{4\cdot3^{k-2}}=-\frac{44224}{59049}=691\cdot\frac{-64}{59049}.$$
\end{itemize}
Notice that the irregular prime $691$ is the one coming from the numerator of the Bernoulli number $B_{12}$.
\end{example}
\begin{remark}
The divisibility of the irregular primes from $B_k$ in the products of those eigenvalues
\begin{equation*}
\prod_{i=1}^{\dim\cS_k(\SL_2(\Z))}\frac{\lambda_{2,i}-(1+2^{k-1})}{2^{k-2}}\textrm{ and }\prod_{i=1}^{\dim\cS_k(\SL_2(\Z))}\frac{\lambda_{3,i}-(1+3^{k-1})}{4\cdot 3^{k-2}}
\end{equation*}
is coming from the well-known result about the congruence between Eisenstein series and cusp forms.
\end{remark}

\begin{example}
When $k=10$, we have
\begin{eqnarray*}
\cD_{10,2}^{2}\cdot \cC_{10,2}^{2}&=&
\begin{displaystyle}\left(\begin{array}{rrr}
-14 & 0 & \frac{53}{4} \\
-6 & -\frac{15}{16} & 6 \\
-\frac{127}{64} & 0 & 1
\end{array}\right)\end{displaystyle}\\
\cD_{10,2}^{3}\cdot \cC_{10,2}^{3}&=&
\begin{displaystyle}\left(\begin{matrix}
-14 & 0 & \frac{122}{9} \\
-6 & -\frac{40}{81} & 6 \\
-\frac{1093}{729} & 0 & 1
\end{matrix}\right)\end{displaystyle}
\end{eqnarray*}
The first matrix has one left eigenvector $(2,-7,8)$ with eigenvalue $-\frac{15}{16}$, and it corresponds to the restricted even period polynomial 
\begin{equation*}
r_f^{\ev,0}(x,y)=2x^2y^6-7x^4y^4+8x^6y^2
\end{equation*}
of the unique newform in $\cS_{10}^{\textrm{new}}(\Gamma_0(2))^-$.
The second matrix has one left eigenvector $(2,-9,18)$ with eigenvalue $-\frac{40}{81}$, and it corresponds to the restricted even period polynomial 
\begin{equation*}
r_f^{\ev,0}(x,y)=2x^2y^6-9x^4y^4+18x^6y^2
\end{equation*}
of the unique newform in $\cS_{10}^{\textrm{new}}(\Gamma_0(3))^-$.
It also has another left eigenvector $(1,0,-9)$ with eigenvalue $-\frac{41}{81}$, and it corresponds to the restricted even period polynomial 
\begin{equation*}
r_f^{\ev,0}(x,y)=x^2y^6-9x^6y^2
\end{equation*}
of the unique newform in $\cS_{10}^{\textrm{new}}(\Gamma_0(3))^+$.
\end{example}

\section{Applications}\label{sc8}
In this section, we will provide some applications about Theorem \ref{hecke} and Theorem \ref{newform}.
From Theorem \ref{hecke}, we know that when $N=2,3$ each eigenform $f\in\cS_{k}(\SL_2(\Z))$ gives us a left eigenvector of $\cC_{k,2}^{N}$ with eigenvalues
\begin{itemize}
\item $N=2$, 
$$\frac{\lambda_2-(1+2^{k-1})}{2^{k-2}},$$
\item $N=3$,
$$\frac{\lambda_3-(1+3^{k-1})}{4\cdot3^{k-2}}.$$
\end{itemize}
We have the following result due to Deligne, which gives an estimation of the coefficients of cusp forms.
\begin{theoremNoNum}[Deligne \cite{D}]
For any $f=\sum_{n=1}^\infty a_nq^n\in\cS_{k}(\SL_2(\Z))$ and prime $p$, we have
$$|a_p|\leq 2p^{\frac{k-1}{2}}.$$
\end{theoremNoNum}
Therefore, we can see the corresponding eigenvalues for eigenforms have the following asymptotical behavior.
\begin{eqnarray*}
\frac{\lambda_2-(1+2^{k-1})}{2^{k-2}}&\sim& -2,\\
\frac{\lambda_3-(1+3^{k-1})}{4\cdot3^{k-2}}&\sim& -\frac{3}{4}.
\end{eqnarray*}
By a numerical computation up to weight $\leq 100$, we found that the number of the eigenvalues being closed to $-2$ ($-\frac{3}{4}$ resp.) is exactly the dimension of space of cusp forms, which suggests that we can use $\cC_{k,2}^{2}$ and $\cC_{k,2}^{3}$ to compute the eigenvalues of the $T_2$ and $T_3$ operators and also to compute the period polynomial of eigenforms. One advantage of using those two matrices is that we have explicit formulas for them which only contain binomial coefficients.

From Theorem \ref{newform}, we know that when $N=2,3$ each $f\in\cS_{k}^{\textrm{new}}(\Gamma_0(N))^{\pm}$ gives us a left eigenvector of $(\cD_{k,2}^{N}\cdot \cC_{k,2}^{N})$ with eigenvalues described by the theorem. By a numerical computation up to weight $\leq 100$, we found that the dimension of eigenspaces match with the dimension of the corresponding newform space, and hence we made the following conjecture. The summation starts from $k=6$, since we already known that $\cS_{<6}^{\textrm{new}}(\Gamma_0(N))=0$ is empty when $N=2,3$, and also our method only works for $k\geq 6$ since $\o\o(k)\neq\varnothing$ when $k\geq 6$.
\begin{conj}
We have the following generating series of the dimensions of $\cS_{k}^{\mathrm{new}}(\Gamma_0(N))^{\pm}$ for $N=2,3$.
\begin{eqnarray}
\sum_{k=6}^\infty \dim(\cS^{\mathrm{new}}_k(\Gamma_0(2)))^+x^k&=&\frac{x^8}{(1-x^6)(1-x^8)},\\
\sum_{k=6}^\infty \dim(\cS^{\mathrm{new}}_k(\Gamma_0(2)))^-x^k&=&\frac{x^2(1+x^{18})}{(1-x^8)(1-x^{12})},\\
\sum_{k=6}^\infty \dim(\cS^{\mathrm{new}}_k(\Gamma_0(3)))^+x^k&=&\frac{x^8}{(1-x^2)(1-x^{12})},\\
\sum_{k=6}^\infty \dim(\cS^{\mathrm{new}}_k(\Gamma_0(3)))^-x^k&=&\frac{x^6(1+x^8+x^{10}-x^{12})}{(1-x^4)(1-x^{12})}.
\end{eqnarray}
\end{conj}

\section*{Acknowledgement}
This study was funded by NSF Grant No. DMS-1401122. The author would like to thank Romyar Sharifi, Masanobu Kaneko and Koji Tasaka for many discussion.

\bibliographystyle{amsplain}

\end{document}